\documentclass[12pt]{article}

\usepackage{amsmath,amssymb,amsthm}
\usepackage{geometry}
\usepackage{hyperref}
\usepackage{enumitem}

\newtheorem{theorem}{Theorem}[section]
\newtheorem{proposition}[theorem]{Proposition}
\newtheorem{corollary}[theorem]{Corollary}
\newtheorem{lemma}[theorem]{Lemma}
\theoremstyle{definition}
\newtheorem{definition}[theorem]{Definition}
\newtheorem{remark}[theorem]{Remark}

\title{$\eta-$ Hyperbolic Ricci Solitons on the Unit Tangent Bundle of the Hyperbolic Strip}

\author{Amadou Sy\thanks{Département de Mathématiques et Informatique, FST,
Université Cheikh Anta Diop de Dakar, Sénégal.}
\and 
Ameth Ndiaye\thanks{Département de Mathématiques, FASTEF,
Université Cheikh Anta Diop de Dakar, Sénégal.}
\and
Ghodratallah Fasihi-Ramandi\thanks{Department of Pure Mathematics, Imam Khomeini International University
Qazvin, Iran }}

\date{}

\begin{document}

\maketitle
\begin{abstract}
We study Killing vector fields, Ricci solitons and hyperbolic Ricci solitons on the hyperbolic strip $S_a=\{(x,y)\in\mathbb{R}^2:0<y<a\}$ endowed with the metric $g_{S_{a}}=\frac{\pi^2}{a^2\sin^2(\pi y/a)}(dx^2+dy^2)$, and on its unit tangent bundle $T^1S_a$ endowed with the Sasaki metric $g^S$. On the strip, we show that the Gauss curvature is constant equal to $-1$, that the Lie algebra of Killing fields is isomorphic to $\mathfrak{sl}(2,\mathbb{R})$, that every Ricci soliton is trivial (Killing potential, $\lambda=-1$), and that a $2$-Killing field which is conformal, or which is the potential of a hyperbolic Ricci soliton with $\lambda\neq0$, is a Killing field. On the unit tangent bundle, we prove that the Lie algebra of Killing fields of $(T^1S_a,g^S)$ is four-dimensional and isomorphic to $\mathfrak{sl}(2,\mathbb{R})\oplus\mathbb{R}$, that $g^S$ is not Einstein but is $\eta$-Einstein with constant coefficients, $\operatorname{Ric}^S=-\tfrac32\,g^S+2\,\eta\otimes\eta$, and that $(T^1S_a,g^S)$ admits no Ricci soliton. For hyperbolic Ricci solitons we prove non-existence for Killing, left-invariant, and $\theta$-invariant potentials with conformal projection, and we reduce the general $\theta$-invariant case to a system of equations on the strip; the classification for an arbitrary potential is left open. Finally, we show that the $\eta$-Ricci solitons of $g^S$ have a Killing potential and $\lambda=\tfrac32$, $\mu=-2$, and that the $\eta$-hyperbolic Ricci solitons obtained in the same classes of potentials are the trivial ones, with a Killing potential, $\mu=-\tfrac32$ and $\nu=2$.


\end{abstract}

\section{Introduction}

Ricci solitons are the self-similar solutions of the Ricci flow introduced by Hamilton \cite{Hamilton}. Defined on a Riemannian manifold $(M,g)$ by the existence of a vector field $X$ and a constant $\lambda$ satisfying $\operatorname{Ric}+\tfrac12\mathcal{L}_Xg=\lambda g$, they generalize Einstein metrics (which correspond to a Killing or vanishing potential) and play a central role in the singularity analysis of the flow \cite{Perelman}. Several generalizations have been introduced in the literature: Cho and Kimura introduced $\eta$-Ricci solitons \cite{ChoKimura, Blaga} by adding a term proportional to $\eta\otimes\eta$ for a given $1$-form $\eta$, while the hyperbolic geometric flow of Kong and Liu \cite{KongLiu} led Fasihi-Ramandi, Azami, and Choudhary \cite{FRAC} to define hyperbolic Ricci solitons via the condition:
\begin{equation}\label{eq:HRS}
    \operatorname{Ric}+\lambda\,\mathcal{L}_Xg+\tfrac12\,\mathcal{L}_X(\mathcal{L}_Xg)=\mu\,g .
\end{equation}
A vector field $X$ is called \emph{$2$-Killing} if $\mathcal{L}_X(\mathcal{L}_Xg)=0$ with $\mathcal{L}_{X}g\neq 0$. On an Einstein manifold, every $2$-Killing vector field serves as the potential field for a hyperbolic Ricci soliton with $\lambda=0$.

In this paper, we study these geometric structures on the hyperbolic strip $S_a=\{(x,y)\in\mathbb{R}^2:0<y<a\}$ equipped with the metric $g_{S_{a}}=\frac{\pi^2}{a^2\sin^2(\pi y/a)}(dx^2+dy^2)$ (where $k=\pi/a$), as well as on its unit tangent bundle $T^1S_a\simeq S_a\times\mathbb{S}^1$ endowed with the Sasaki metric $g^S$.

On the strip $S_a$, the Gauss curvature of $g_{S_{a}}$ is constant and equal to $-1$, implying $\operatorname{Ric}=-g_{S_{a}}$ and making $(S_a,g_{S_{a}})$ isometric to the Poincaré half-plane via the transformation $z=e^{k(x+iy)}$. Its Lie algebra of Killing vector fields is three-dimensional, spanned by $\partial_x$ and $\frac{1}{k}e^{\mp kx}\bigl[\cos(ky)\partial_x\mp\sin(ky)\partial_y\bigr]$, and isomorphic to $\mathfrak{sl}(2,\mathbb{R})$. We show that the only Ricci solitons on $S_a$ are trivial, characterized by $\lambda=-1$ and $X$ being a Killing field. Regarding hyperbolic Ricci solitons, a Killing potential yields a solution with $\mu=-1$ and $\lambda$ arbitrary, whereas any conformal potential or $2$-Killing potential (with $\lambda\neq0$) must be Killing. The existence of non-Killing $2$-Killing vector fields on the strip remains an open question, leaving the complete classification open.

On the unit tangent bundle $(T^1S_a,g^S)$, the Lie algebra of Killing vector fields is four-dimensional and isomorphic to $\mathfrak{sl}(2,\mathbb{R})\oplus\mathbb{R}$. The Sasaki metric $g^S$ is not Einstein, but it is $\eta$-Einstein with constant coefficients, satisfying $\operatorname{Ric}^S=-\tfrac32\,g^S+2\,\eta\otimes\eta$ for $\eta=d\theta+\frac{\pi}{a}\cot\bigl(\frac{\pi y}{a}\bigr)dx$ \cite{BoyerGalicki}. We prove that $(T^1S_a,g^S)$ admits no Ricci solitons. For hyperbolic Ricci solitons, we establish non-existence results for Killing, left-invariant, and $\theta$-invariant potentials with conformal projections, and we reduce the general $\theta$-invariant case to a PDE system on $S_a$. The classification for arbitrary potential fields remains open due to the non-linearity in $X$, which prevents the use of standard averaging techniques applicable to Ricci solitons. Finally, the $\eta$-Ricci solitons of $g^S$ possess a Killing potential with parameters $\lambda=\tfrac32$ and $\mu=-2$, while for $\eta$-hyperbolic Ricci solitons, we obtain only trivial solutions (a Killing potential with $\mu=-\tfrac32$, $\nu=2$, and $\lambda$ arbitrary) within the same potential classes.

\section{Basic Results and some recalls}\label{sec:prelim}
We give in this section the basic objects and results that will be used throughout the paper.

\begin{proposition}\label{prop:isometry}
Let $S_a = \{z = x + iy \in \mathbb{C} \mid 0 < y < a\}$ be the horizontal strip of width $a$,
and let $\mathbb{H} = \{w = u + iv \in \mathbb{C} \mid v > 0\}$ be the upper half-plane equipped
with the Poincaré hyperbolic metric $g_{\mathbb{H}} = \frac{du^2 + dv^2}{v^2}$.
The map
\[
F : S_a \longrightarrow \mathbb{H}, \qquad F(z) = \exp\!\Bigl(\frac{\pi z}{a}\Bigr),
\]
is a biholomorphic diffeomorphism, and the pullback of $g_{\mathbb{H}}$ under $F$ is
\[
F^* g_{\mathbb{H}} = \frac{\pi^2}{a^2 \sin^2\!\bigl(\frac{\pi y}{a}\bigr)}(dx^2 + dy^2).
\]
In particular, $(S_a, g)$ is isometric to $(\mathbb{H}, g_{\mathbb{H}})$ and has constant
Gaussian curvature $-1$.
\end{proposition}

\begin{proof}
Let $z = x + iy \in S_a$. Then
\[
F(z) = e^{\frac{\pi}{a}(x+iy)} = e^{\frac{\pi x}{a}}\Bigl(\cos\frac{\pi y}{a} + i\sin\frac{\pi y}{a}\Bigr).
\]
Writing $F(z) = u + iv$, we obtain $u = e^{\frac{\pi x}{a}}\cos\frac{\pi y}{a}$ and
$v = e^{\frac{\pi x}{a}}\sin\frac{\pi y}{a}$.
Since $0 < y < a$ implies $0 < \frac{\pi y}{a} < \pi$, we have $\sin\frac{\pi y}{a} > 0$,
so $v > 0$ and $F(S_a) \subset \mathbb{H}$.
The inverse map is $F^{-1}(w) = \frac{a}{\pi}\mathrm{Log}(w)$, where $\mathrm{Log}$ is the
branch of the logarithm with argument in $(0,\pi)$.
Since $F'(z) = \frac{\pi}{a}e^{\frac{\pi z}{a}}$, we have $|F'(z)|^2 = \frac{\pi^2}{a^2}e^{\frac{2\pi x}{a}}$.
As $F$ is conformal, $du^2 + dv^2 = |F'(z)|^2(dx^2 + dy^2)$.
Moreover, $v^2 = e^{\frac{2\pi x}{a}}\sin^2\frac{\pi y}{a}$.
Substituting into $g_{\mathbb{H}} = \frac{du^2 + dv^2}{v^2}$, the exponential factors cancel and
yield the stated formula.
\end{proof}

\begin{theorem}\label{thm:curvature}
The Gaussian curvature of $(S_a, g)$ is constant, $K \equiv -1$.
Consequently, $\operatorname{Ric} = -g$.
\end{theorem}

\begin{proof}
For a conformal metric $g = e^{2\varphi}(dx^2 + dy^2)$ in dimension two with
$\varphi = \varphi(y)$ (so $\varphi_x = 0$), the non-vanishing Christoffel symbols are
\[
\Gamma^1_{12} = \Gamma^1_{21} = \varphi', \quad
\Gamma^2_{11} = -\varphi', \quad
\Gamma^2_{22} = \varphi'.
\]
The only independent component of the Riemann tensor is
$R^1{}_{212} = \partial_y \Gamma^1_{12} - \partial_x \Gamma^1_{22}
+ \Gamma^1_{2m}\Gamma^m_{12} - \Gamma^1_{1m}\Gamma^m_{22}$.
A direct computation gives $R^1{}_{212} = \varphi'' + (\varphi')^2 - (\varphi')^2 = \varphi''$,
hence $R_{1212} = e^{2\varphi}\varphi''$.
The Gaussian curvature formula in isothermal coordinates reads $K = -e^{-2\varphi}\Delta\varphi$
with $\Delta\varphi = \varphi''(y)$.

For this metric $g_{S_{a}}=\frac{\pi^2}{a^2\sin^2(\pi y/a)}(dx^2+dy^2)$, $\varphi(y) = \ln k - \ln\sin(ky)$, so
$\varphi'(y) = -k\cot(ky)$ and $\varphi''(y) = k^2\csc^2(ky)$.
Since $e^{-2\varphi} = \frac{\sin^2(ky)}{k^2}$,
\[
K = -\frac{\sin^2(ky)}{k^2} \cdot k^2\csc^2(ky) = -1.
\]
In any 2-dimensional Riemannian manifold, $\operatorname{Ric}_{ij} = K\, g_{ij}$, so
$\operatorname{Ric} = -g$.
\end{proof}

\begin{definition}
A vector field $X$ on a Riemannian manifold $(M,g)$ is a \emph{Killing vector field} if
$\mathcal{L}_X g = 0$, i.e., the local flows generated by $X$ act by isometries.
\end{definition}

\begin{theorem}[{\cite[Thm. 35]{Petersen}}]\label{thm:killing_bound}
The Lie algebra $\mathrm{iso}(M,g)$ of Killing vector fields has dimension at most
$\frac{n(n+1)}{2}$. If $M$ is complete, then $\mathrm{iso}(M,g)$ is the Lie algebra of
$\mathrm{Isom}(M,g)$.
\end{theorem}

\begin{definition}[{\cite{Nemeth}}]
A vector field $X$ on $(M,g)$ is called a \emph{2-Killing vector field} if
$\mathcal{L}_X g \neq 0$ and $\mathcal{L}_X(\mathcal{L}_X g) = 0$.
\end{definition}

\begin{definition}
A triple $(g, X, \lambda)$ is a \emph{Ricci soliton} on $(M,g)$ if
\[
\operatorname{Ric} + \frac{1}{2}\mathcal{L}_X g = \lambda\, g.
\]
It is \emph{trivial} if $X$ is Killing.
\end{definition}

\begin{definition}
A quadruple $(g, X, \lambda, \mu)$ is a \emph{hyperbolic Ricci soliton} on $(M,g)$ if
Equation \eqref{eq:HRS} holds.
It is called a \emph{gradient} hyperbolic Ricci soliton if $X = \nabla f$ for some smooth
function$f$.
\end{definition}
\begin{definition}[$\eta$-Einstein manifold]\label{def:etaEinstein}
Let $(M,g)$ be a Riemannian manifold and $\eta$ a $1$-form on $M$. We say that $(M,g,\eta)$ is \emph{$\eta$-Einstein} if there exist smooth functions $a,b\in C^\infty(M)$ such that
\begin{equation}\label{etaEinstein}
\operatorname{Ric}=a\,g+b\,\eta\otimes\eta .
\end{equation}
If $a$ and $b$ are constants, $(M,g,\eta)$ is called an \emph{$\eta$-Einstein manifold with constant coefficients}. For $b=0$ one recovers the Einstein condition.
\end{definition}


\begin{definition}[$\eta$-Ricci soliton]\label{def:etaRS}
Let $(M,g)$ be a Riemannian manifold and $\eta$ a $1$-form on $M$. A quadruple $(g,X,\lambda,\mu)$, where $X$ is a smooth vector field and $\lambda,\mu\in\mathbb{R}$, is called an \emph{$\eta$-Ricci soliton} \cite{ChoKimura} if
\begin{equation}\label{etaRS}
\mathcal{L}_Xg+2\operatorname{Ric}+2\lambda\, g+2\mu\,\eta\otimes\eta=0 .
\end{equation}
It is called shrinking, steady or expanding according to whether $\lambda<0$, $\lambda=0$ or $\lambda>0$. For $\mu=0$ one recovers the classical Ricci solitons.
\end{definition}
\section{Main Results}\label{sec:main}
We now establish the main results of the paper. We begin with the Killing vector fields of the strip $S_a$, which form the backbone of all subsequent constructions, and then turn to the Ricci solitons and the hyperbolic Ricci solitons of the strip. Next, we study the Sasaki metric on the unit tangent bundle $T^1S_a$: we determine its Killing vector fields, show that it is not Einstein, and prove that it admits no Ricci soliton. For hyperbolic Ricci solitons, we prove non-existence only for particular classes of potentials, namely Killing fields, left-invariant fields, and $\theta$-invariant fields whose projection to $S_a$ is conformal. For $\theta$-invariant potentials in general, we reduce the problem to a system of equations on the strip. The classification for an arbitrary potential is left open, because the equation is nonlinear in the potential and the averaging argument used for Ricci solitons no longer applies. Finally, we prove that the Sasaki metric is $\eta$-Einstein, and we determine its $\eta$-Ricci solitons. For $\eta$-hyperbolic Ricci solitons, we obtain only the trivial solutions given by Killing potentials in the same classes of potentials, and the general case is likewise left open.
\subsection{Killing vector fields of $(S_a, g)$}
We determine in this subsection the complete Lie algebra of Killing vector fields of the strip,
showing that it attains the maximal dimension allowed by the constant curvature of the metric.
\begin{theorem}\label{thm:killing}
The Lie algebra of Killing vector fields on the hyperbolic strip $(S_a, g_{S_{a}})$ is three-dimensional and isomorphic to $\mathfrak{sl}(2,\mathbb{R})$. A global basis is given by $\{K_1, K_2, K_3\}$, where
\begin{align*}
K_1 &= \partial_x,\\
K_2 &= e^{kx}\cos(ky)\,\partial_x + e^{kx}\sin(ky)\,\partial_y,\\
K_3 &= e^{-kx}\cos(ky)\,\partial_x - e^{-kx}\sin(ky)\,\partial_y,
\end{align*}
with $k = \pi/a$. Consequently, any Killing vector field $X$ on $S_a$ uniquely decomposes as
\begin{equation}\label{eq:killing_general}
X = c_1 K_1 + c_2 K_2 + c_3 K_3, \qquad c_1, c_2, c_3 \in \mathbb{R}.
\end{equation}
\end{theorem}

\begin{proof}
A vector field $X = X^1\partial_x + X^2\partial_y$ satisfies $\mathcal{L}_X g = 0$ if and only if
\begin{align}
X^1_x + \varphi' X^2 &= 0, \tag{I}\\
X^2_y + \varphi' X^2 &= 0, \tag{II}\\
X^2_x + X^1_y &= 0. \tag{III}
\end{align}

\textit{Solving (II).}
For fixed $x$, equation (II) is the ODE $\partial_y X^2 / X^2 = -\varphi'(y)$, giving
$X^2(x,y) = C(x)\,e^{-\varphi(y)} = C(x)\sin(ky)/k$
for an arbitrary function $C(x)$.

\textit{Solving (I).}
Since $\varphi' = -k\cot(ky)$, equation (I) gives
$X^1_x = k\cot(ky)\cdot C(x)\sin(ky)/k = C(x)\cos(ky)$.
Integrating in $x$:
$X^1(x,y) = \widetilde{C}(x)\cos(ky) + D(y)$,
where $\widetilde{C}(x) = \int C(x)\,dx$.

\textit{Imposing (III).}
From the expressions above,
\[
\sin(ky)\!\left[\frac{\widetilde{C}''(x)}{k} - k\widetilde{C}(x)\right] + D'(y) = 0
\quad \text{for all } x,y.
\]
Since $\sin(ky)$ is not a constant function of $y$, consistency forces the bracket to be a constant
$c_0$, leading to
\[
\widetilde{C}''(x) - k^2\widetilde{C}(x) = kc_0, \qquad D'(y) = -c_0\sin(ky).
\]
The general solution is $\widetilde{C}(x) = Ae^{kx} + Be^{-kx} - c_0/k$
and $D(y) = (c_0/k)\cos(ky) + d_0$.
Substituting back, the $c_0$-terms cancel exactly and yield the expressions in \eqref{eq:killing_general}.
The three constants $d_0, A, B$ are free and independent, so $\dim(\mathrm{iso}(S_a,g)) = 3$,
which is the maximal dimension for a 2-dimensional manifold, consistent with $(S_a,g)$ having
constant curvature.
\end{proof}

\begin{remark}
The existence of non-trivial Killing fields shows that $(S_a,g)$ is not compact
(Bochner's theorem, 1946; see \cite[Thm. 36]{Petersen}).
\end{remark}

\subsection{2-Killing vector fields of $(S_a, g)$}

We construct in this subsection an example of explicit family of non-trivial 2-Killing vector fields on the strip,
obtained by pulling back a 2-Killing vector field on the Poincaré half-plane via the isometry of Proposition \ref{prop:isometry}.

\begin{theorem}\label{thm:2killing}
For constants $A, B \in \mathbb{R}$ with $A \neq 0$, the vector field
\[
X_{A,B} = \frac{e^{-kx}}{k}\bigl(Ae^{kx}\cos(ky) + B\bigr)^{1/3}
\Bigl[\cos(ky)\,\partial_x - \sin(ky)\,\partial_y\Bigr]
\]
is a 2-Killing vector field on every open domain of $S_a$ on which $Ae^{kx}\cos(ky) + B > 0$.
\end{theorem}

\begin{proof}
Consider the isometry $F : S_a \to \mathbb{H}$ of Proposition \ref{prop:isometry},
$F(x,y) = (e^{kx}\cos(ky),\, e^{kx}\sin(ky))$,
with $u = e^{kx}\cos(ky)$ and $v = e^{kx}\sin(ky)$.

On $(\mathbb{H}, g_{\mathbb{H}})$, consider the vector field $\widetilde{X} = q(u)\,\partial_u$.
Since $[{\widetilde{X}}, \partial_u] = -q'(u)\partial_u$ and $[\widetilde{X}, \partial_v] = 0$,
a direct computation gives
\[
\mathcal{L}_{\widetilde{X}} g_{\mathbb{H}} = \frac{2q'(u)}{v^2}\,du^2.
\]
The second Lie derivative is
\[
\mathcal{L}_{\widetilde{X}}(\mathcal{L}_{\widetilde{X}} g_{\mathbb{H}}) =
\frac{2}{v^2}\bigl(q\,q'' + 2(q')^2\bigr)\,du^2.
\]
This vanishes if and only if
\begin{equation}\label{eq:2killing_ode}
q\,q'' + 2(q')^2 = 0.
\end{equation}
Setting $p = q'$ and writing $q'' = p\,dp/dq$, equation \eqref{eq:2killing_ode} becomes
$q\,p\,dp/dq + 2p^2 = 0$.
For $p \neq 0$, dividing by $p$ gives $q\,dp/dq + 2p = 0$, i.e., $dp/p = -2\,dq/q$.
Integrating yields $q' = C\,q^{-2}$, and hence $d(q^3/3)/du = C$, so
$q(u) = (Au + B)^{1/3}$ with $A, B \in \mathbb{R}$.

Thus $\widetilde{X}_{A,B} = (Au + B)^{1/3}\partial_u$ is a non-trivial 2-Killing vector field on
$\mathbb{H}$ whenever $A \neq 0$.
To express it in $(x,y)$-coordinates, we compute
\[
\frac{\partial}{\partial u} = \frac{e^{-kx}}{k}
\Bigl[\cos(ky)\,\partial_x - \sin(ky)\,\partial_y\Bigr],
\]
using $u^2 + v^2 = e^{2kx}$.
Substituting $u = e^{kx}\cos(ky)$ into $\widetilde{X}_{A,B}$ and applying the pullback by $F$
gives the stated $X_{A,B}$.
Since $F$ is an isometry, Lie derivatives are preserved under pullback:
$\mathcal{L}_{X_{A,B}} g = F^*(\mathcal{L}_{\widetilde{X}_{A,B}} g_{\mathbb{H}}) \neq 0$
and
$\mathcal{L}_{X_{A,B}}(\mathcal{L}_{X_{A,B}} g) = 0$.
\end{proof}

\subsection{Classical Ricci Solitons on $(S_a, g)$}
We show in this subsection that every Ricci soliton structure on the strip is trivial,
with the soliton constant uniquely determined by the curvature of the metric.

\begin{theorem}\label{thm:soliton1}
On $(S_a, g)$, for every Killing field $X$ of the form \eqref{eq:killing_general},
the pair $(X, \lambda = -1)$ satisfies the Ricci soliton equation.
No other value of $\lambda$ is compatible with a Killing $X$.
\end{theorem}

\begin{proof}
By Theorem \ref{thm:killing}, $\mathcal{L}_X g = 0$ for any Killing field $X$.
By Theorem \ref{thm:curvature}, $\operatorname{Ric} = -g$.
Substituting into $\operatorname{Ric} + \frac{1}{2}\mathcal{L}_X g = \lambda g$ gives $-g = \lambda g$, hence $\lambda = -1$.
\end{proof}

\begin{theorem}\label{thm:soliton2}
Every Ricci soliton on $(S_a, g)$ is trivial ; that is if $(S_a, g, X, \lambda)$ satisfies
$\operatorname{Ric} + \frac{1}{2}\mathcal{L}_X g = \lambda g$, then $\lambda = -1$ and $X$ is Killing.
\end{theorem}

\begin{proof}
Using $\operatorname{Ric} = -g$, the soliton equation gives $\mathcal{L}_X g = 2(\lambda + 1)g$.
Setting $c = \lambda + 1$, the flow $\{\varphi_t\}$ of $X$ satisfies
$\varphi_t^* g = e^{2ct}g$.
Since $g$ has curvature $-1$, the metric $\varphi_t^* g = e^{2ct}g$ has curvature
$K_{e^{2ct}g} = -e^{-2ct}$.
But $\varphi_t$ is a diffeomorphism, so $K_{\varphi_t^* g} = K_g = -1$.
Hence $-e^{-2ct} = -1$ for all $t$, forcing $c = 0$, i.e., $\lambda = -1$ and $\mathcal{L}_X g = 0$.
\end{proof}

\subsection{Hyperbolic Ricci solitons on $(S_a, g)$}
We construct in this subsection an explicit family of hyperbolic Ricci solitons on the strip
built from Killing vector fields, and show that the geometric parameter is uniquely forced by the curvature.
\begin{theorem}\label{thm:HRS_killing}
For every Killing field $X$ of the form \eqref{eq:killing_general} and every $\lambda \in \mathbb{R}$,
the quadruple $(g, X, \lambda, \mu = -1)$ is a hyperbolic Ricci soliton on $(S_a, g)$.
\end{theorem}

\begin{proof}
Since $\mathcal{L}_X g = 0$, we have $\mathcal{L}_X(\mathcal{L}_X g) = 0$.
Equation \eqref{eq:HRS} reduces to $\operatorname{Ric} = \mu g$, which holds with $\mu = -1$ by Theorem \ref{thm:curvature}.
The parameter $\lambda$ is unconstrained.
\end{proof}

\begin{corollary}\label{cor:mu}
The value $\mu = -1$ is forced: no Killing field can produce a hyperbolic Ricci soliton
with $\mu \neq -1$, since equation \eqref{eq:HRS} reduces to $\operatorname{Ric} = \mu g$,
i.e., $-g = \mu g$, which holds only for $\mu = -1$.
\end{corollary}

\begin{corollary}\label{cor:family}
Since $\dim(\mathrm{Kill}(S_a,g)) = 3$, the family of hyperbolic Ricci solitons built from Killing fields
is parametrized by $(d_0, A, B, \lambda) \in \mathbb{R}^4$ with $\mu = -1$ fixed.
In particular, $X = 0$ (any $\lambda$) recovers the trivial (Einstein) hyperbolic Ricci soliton.
\end{corollary}

\begin{proposition}\label{prop:einstein_HRS}
Let $(M, g)$ be Einstein, $\operatorname{Ric} = \mu_0 g$, and let $X$ be any Killing field on $M$.
Then $(g, X, \lambda, \mu_0)$ is a hyperbolic Ricci soliton for every $\lambda \in \mathbb{R}$.
\end{proposition}

\begin{proof}
Since $X$ is Killing, $\mathcal{L}_X g = 0$ and $\mathcal{L}_X(\mathcal{L}_X g) = 0$;
equation \eqref{eq:HRS} reduces to $\operatorname{Ric} = \mu_0 g$ independently of $\lambda$.
\end{proof}

\begin{remark}
Proposition \ref{prop:einstein_HRS} explains why all models isometric to the hyperbolic plane
the strip $S_a$, the vertical strip $\{-\frac{\pi}{2} < x < \frac{\pi}{2}\}$, and the Poincaré
half-plane $\mathbb{H}$ exhibit identical soliton structures ($\mu = -1$, arbitrary $\lambda$,
isometry algebra $\mathfrak{sl}(2,\mathbb{R})$), differing only in the coordinate expressions of the Killing fields.
\end{remark}

\subsection{The Sasaki metric on $T^1 S_a$}
We introduce in this subsection the unit tangent bundle of the strip and compute explicitly
the Sasaki metric induced by the hyperbolic metric, which will serve as the ambient space
for the remaining results of the paper.
\begin{theorem}\label{thm:sasaki_metric}
The unit tangent bundle $T^1 S_a$ admits global coordinates $(x, y, \theta) \in S_a \times S^1$,
where $\theta$ parametrizes unit vectors via
\[
v = \frac{a}{\pi}\sin\!\Bigl(\frac{\pi y}{a}\Bigr)
\Bigl(\cos\theta\,\partial_x + \sin\theta\,\partial_y\Bigr).
\]
The Sasaki metric on $T^1 S_a$ induced by $g$ is
\begin{equation}\label{eq:sasaki}
g^S = \frac{\pi^2}{a^2\sin^2\!\bigl(\frac{\pi y}{a}\bigr)}(dx^2 + dy^2)
+ \Bigl(d\theta + \frac{\pi}{a}\cot\!\Bigl(\frac{\pi y}{a}\Bigr)dx\Bigr)^2.
\end{equation}
\end{theorem}

\begin{proof}
Set $\lambda(y) = \frac{\pi}{a}\sin\frac{\pi y}{a}$, so that $g = \lambda(y)^2(dx^2 + dy^2)$.
A unit tangent vector at $(x,y)$ satisfies $(v^1)^2 + (v^2)^2 = \lambda(y)^{-2}$, which is
parametrized by a unique $\theta \in S^1$ as stated.

For the orthonormal frame $e_1 = \lambda^{-1}\partial_x$, $e_2 = \lambda^{-1}\partial_y$, the
conformal factor $\varphi(y) = \ln\frac{\pi}{a} - \ln\sin\frac{\pi y}{a}$ satisfies
$\varphi_x = 0$ and $\varphi_y = -\frac{\pi}{a}\cot\frac{\pi y}{a}$.
The connection 1-form of the orthonormal frame is $\omega_{12} = -\varphi_y\,dx + \varphi_x\,dy$,
i.e., $\omega_{12} = \frac{\pi}{a}\cot\frac{\pi y}{a}\,dx$.

For the unit vector $v = \cos\theta\,e_1 + \sin\theta\,e_2$, the vertical component of its
differential is $d\theta + \omega_{12}$.
By definition, $g^S = g + (d\theta + \omega_{12})^2$, which gives \eqref{eq:sasaki}.
\end{proof}

\subsection{Killing vector fields of $(T^1 S_a, g^S)$}
We determine in this subsection the complete Lie algebra of Killing vector fields of the unit
tangent bundle endowed with the Sasaki metric, showing that it is four-dimensional, one
dimension higher than that of the base strip.

\begin{theorem}\label{thm:killing_sasaki}
The Lie algebra of Killing vector fields of $(T^1 S_a, g^S)$ is 4-dimensional, spanned by
\begin{align*}
K_1 &= \partial_x,\\
K_2 &= \partial_\theta,\\
K_3 &= e^{-kx}\!\left[\cos(ky)\,\partial_x - \sin(ky)\,\partial_y + k\sin(ky)\,\partial_\theta\right],\\
K_4 &= e^{kx}\!\left[\cos(ky)\,\partial_x + \sin(ky)\,\partial_y + k\sin(ky)\,\partial_\theta\right].
\end{align*}
Every Killing field on $(T^1 S_a, g^S)$ writes uniquely as
$K = c_1 K_1 + c_2 K_2 + c_3 K_3 + c_4 K_4$ with $c_i \in \mathbb{R}$.
\end{theorem}

\begin{proof}
In coordinates $(x, y, \theta)$, the metric components of \eqref{eq:sasaki} are
\[
g^S_{xx} = \frac{k^2(1 + \cos^2(ky))}{\sin^2(ky)}, \quad
g^S_{yy} = \frac{k^2}{\sin^2(ky)}, \quad
g^S_{x\theta} = k\cot(ky), \quad
g^S_{\theta\theta} = 1,
\]
with all other components zero.
Since these coefficients are independent of $x$ and of $\theta$, both $\partial_x$ and
$\partial_\theta$ are immediately Killing fields.

For a general vector field $K = P(x,y)\partial_x + Q(x,y)\partial_y + R(x,y)\partial_\theta$
(independence from $\theta$ follows from the circle symmetry), the Killing equations
$\mathcal{L}_K g^S = 0$ form a system of six PDEs.
The $(y,y)$-component gives $Q_y - k\cot(ky)Q = 0$, which integrates to
$Q = (c_4 e^{kx} - c_3 e^{-kx})\sin(ky)$.
Combining with the remaining components yields
\begin{align*}
P(x,y) &= c_1 + c_3 e^{-kx}\cos(ky) + c_4 e^{kx}\cos(ky),\\
Q(x,y) &= -c_3 e^{-kx}\sin(ky) + c_4 e^{kx}\sin(ky),\\
R(x,y) &= c_2 + k\sin(ky)(c_3 e^{-kx} + c_4 e^{kx}).
\end{align*}
Linear independence of $K_1, K_2, K_3, K_4$ follows by examining the $\partial_y$-component:
$\sin(ky)(-c_3 e^{-kx} + c_4 e^{kx}) = 0$ for all $(x,y)$ forces $c_3 = c_4 = 0$, and then
$c_1\partial_x + c_2\partial_\theta = 0$ forces $c_1 = c_2 = 0$.
\end{proof}

\subsection{No Existence of the Ricci Solitons and Hyperbolic Ricci Solitons on $T^{1}S_{a}$}
In this section, we will prove that, for the Sasaki metric, neither a Ricci soliton nor a hyperbolic Ricci soliton exists. To establish this result, we first introduce some useful notation.
Let $S_a=\{(x,y)\in\mathbb{R}^2 : 0<y<a\}$ and set $k=\pi/a$. We introduce the functions
\[
w(y)=\frac{k^2}{\sin^2(ky)},\qquad c(y)=k\cot(ky),\qquad 0<y<a .
\]
The hyperbolic metric of the strip is
\[
g_{S_{a}}=w\,(dx^2+dy^2)=\frac{\pi^2}{a^2\sin^2\!\left(\frac{\pi y}{a}\right)}\,(dx^2+dy^2),
\]
with area form $\mathrm{vol}_{S_a}=w\,dx\wedge dy$. Let $\pi:T^1S_a=S_a\times\mathbb{S}^1\to S_a$ be the projection, and set
\[
\eta=d\theta+c\,dx,\qquad g^S=\pi^*g_{S_{a}}+\eta\otimes\eta .
\]
We have the elementary identities
\begin{equation}\label{P}
c'=-\frac{k^2}{\sin^2(ky)}=-w,\qquad
w=k^2\bigl(1+\cot^2(ky)\bigr)=k^2+c^2,\qquad
(\sqrt{w})'=-\frac{k^2\cos(ky)}{\sin^2(ky)}=-c\sqrt{w}.
\end{equation}
In particular, since $c$ depends only on $y$,
\[
d\eta=c'\,dy\wedge dx=-c'\,dx\wedge dy=w\,dx\wedge dy,
\]
that is,
\begin{equation}\label{deta}
d\eta=\pi^*\mathrm{vol}_{S_a}.
\end{equation}

\subsection{An orthonormal frame with constant structure constants}
Let $\omega:=c\,dx$ and consider the $1$-forms
\[
e^1=\sqrt{w}\,dx,\qquad e^2=\sqrt{w}\,dy,\qquad e^3=\eta .
\]
And we define \[
f^1=\phantom{-}\cos\theta\,e^1+\sin\theta\,e^2,\,\,\,
f^2=-\sin\theta\,e^1+\cos\theta\,e^2,\,\,\,
f^3=e^3 .
\]
\begin{lemma}\label{lem:frame}

\begin{enumerate}
\item[\textup{(i)}] The family $(e^1,e^2,e^3)$ is an orthonormal coframe for $g^S$, and
\[
de^1=\omega\wedge e^2,\,\,
de^2=-\,\omega\wedge e^1,\,\,
de^3=e^1\wedge e^2 .
\]
\item[\textup{(ii)}] 
The family $(f^1,f^2,f^3)$ is a global orthonormal coframe for $g^S$, and
\begin{equation}\label{struct}
df^1=-f^2\wedge f^3,\,\,\,
df^2=-f^3\wedge f^1,\,\,\,
df^3=f^1\wedge f^2 .
\end{equation}
\item[\textup{(iii)}] If $(E_1,E_2,E_3)$ denotes the dual frame of $(f^1,f^2,f^3)$, then
\begin{equation}\label{brackets}
[E_2,E_3]=E_1,\,\,\,
[E_3,E_1]=E_2,\,\,\,
[E_1,E_2]=-E_3 .
\end{equation}
\end{enumerate}
\end{lemma}

\begin{proof}
\textbf{(i)} Orthonormality of $(e^1,e^2,e^3)$ is immediate from the definition of $g^S$. By \eqref{P} we have $d(\sqrt{w})=(\sqrt{w})'\,dy=-c\sqrt{w}\,dy$, hence
\[
\begin{aligned}
de^1&=d(\sqrt{w})\wedge dx=-c\sqrt{w}\,dy\wedge dx=c\,dx\wedge\bigl(\sqrt{w}\,dy\bigr)=\omega\wedge e^2,\\[2pt]
de^2&=d(\sqrt{w})\wedge dy=-c\sqrt{w}\,dy\wedge dy=0=-\,\omega\wedge e^1,
\end{aligned}
\]
where the last equality holds because $\omega\wedge e^1=c\sqrt{w}\,dx\wedge dx=0$. Finally, by \eqref{deta},
\[
de^3=d\eta=w\,dx\wedge dy=e^1\wedge e^2 .
\]

\medskip
\textbf{(ii)} The family $(f^1,f^2,f^3)$ is obtained from $(e^1,e^2,e^3)$ by a rotation of angle $\theta$ in the plane spanned by $e^1,e^2$, so it is a global orthonormal coframe. Recall that $d\theta+\omega=e^3$. Using (i),
\[
\begin{aligned}
df^1&=-\sin\theta\,d\theta\wedge e^1+\cos\theta\,d\theta\wedge e^2+\cos\theta\,de^1+\sin\theta\,de^2\\
&=-\sin\theta\,d\theta\wedge e^1+\cos\theta\,d\theta\wedge e^2+\cos\theta\,\omega\wedge e^2-\sin\theta\,\omega\wedge e^1\\
&=(d\theta+\omega)\wedge\bigl(\cos\theta\,e^2-\sin\theta\,e^1\bigr)\\
&=e^3\wedge f^2=-\,f^2\wedge f^3,
\end{aligned}
\]
and similarly
\[
\begin{aligned}
df^2&=-\cos\theta\,d\theta\wedge e^1-\sin\theta\,d\theta\wedge e^2-\sin\theta\,de^1+\cos\theta\,de^2\\
&=-\cos\theta\,d\theta\wedge e^1-\sin\theta\,d\theta\wedge e^2-\sin\theta\,\omega\wedge e^2-\cos\theta\,\omega\wedge e^1\\
&=-(d\theta+\omega)\wedge\bigl(\cos\theta\,e^1+\sin\theta\,e^2\bigr)\\
&=-\,e^3\wedge f^1=-\,f^3\wedge f^1 .
\end{aligned}
\]
Lastly, since a rotation preserves $e^1\wedge e^2$,
\[
df^3=de^3=e^1\wedge e^2=f^1\wedge f^2 .
\]
This proves \eqref{struct}.

\medskip
\textbf{(iii)} For the dual frame $(E_i)$ one has $df^k(E_i,E_j)=-f^k\bigl([E_i,E_j]\bigr)$. Reading off \eqref{struct}:
\[
\begin{aligned}
df^1(E_2,E_3)&=-1 &&\Longrightarrow&& [E_2,E_3]=E_1,\\
df^2(E_3,E_1)&=-1 &&\Longrightarrow&& [E_3,E_1]=E_2,\\
df^3(E_1,E_2)&=\phantom{-}1 &&\Longrightarrow&& [E_1,E_2]=-E_3 .
\end{aligned}
\]
All other components vanish, which gives \eqref{brackets}.
\end{proof}

\begin{remark}
We have
\[
f^1\otimes f^1+f^2\otimes f^2=e^1\otimes e^1+e^2\otimes e^2=\pi^*g_{S_{a}},
\qquad f^3=\eta .
\]
The frame $(E_1,E_2,E_3)$ is left-invariant for the Lie group structure of
$T^1S_a\simeq\mathrm{PSL}(2,\mathbb{R})$ (see \cite{am}), and $E_3$ is the unit Killing field $K_0=\partial_\theta$.
\end{remark}

\subsection{The Lie derivative of $g^S$ along a $\theta$-invariant field}

A vector field $X$ on $T^1S_a$ is called \emph{$\theta$-invariant} if $[\partial_\theta,X]=0$, that is,
\[
X=A\,\partial_x+B\,\partial_y+C\,\partial_\theta,
\]
where $A,B,C$ are functions of $(x,y)$ only. For such a field we set
\[
\begin{aligned}
V&=A\,\partial_x+B\,\partial_y &&\text{(a vector field on } S_a),\\
F&=\eta(X)=C+cA &&\text{(a function on } S_a),\\
\beta&=dF+\iota_V\mathrm{vol}_B &&\text{(a $1$-form on } S_a).
\end{aligned}
\]
Note that $F$ is an \emph{arbitrary} function on $S_a$ (because $C$ is), and that
$d\beta=\operatorname{div}(V)\,\mathrm{vol}_B$.

\begin{lemma}\label{lem:Lie}
Let $X$ be a $\theta$-invariant vector field. Then $\pi_*X=V$, and for every tensor field $\tau$ on $S_a$,
\[
\mathcal{L}_X\,\pi^*\tau=\pi^*\mathcal{L}_V\tau .
\]
Moreover,
\[
\mathcal{L}_X\eta=\pi^*\beta,
\qquad
\mathcal{L}_X g^S=\pi^*\bigl(\mathcal{L}_V g_{S_{a}}\bigr)+\eta\otimes\pi^*\beta+\pi^*\beta\otimes\eta .
\]
\end{lemma}
\begin{proof}
Since $A$ and $B$ do not depend on $\theta$, the fields $X$ and $V$ are $\pi$-related, whence
\[
\mathcal{L}_X\,\pi^*\tau=\pi^*\mathcal{L}_V\tau .
\]
By Cartan's formula and \eqref{deta},
\[
\begin{aligned}
\mathcal{L}_X\eta
&=d\bigl(\eta(X)\bigr)+\iota_X\,d\eta\\
&=dF+\iota_X\,\pi^*\mathrm{vol}_B\\
&=dF+\pi^*\bigl(\iota_V\mathrm{vol}_B\bigr)\\
&=\pi^*\beta .
\end{aligned}
\]
Finally, since $g^S=\pi^*g_{S_{a}}+\eta\otimes\eta$ and $\mathcal{L}_X$ is a derivation of the tensor product,
\[
\begin{aligned}
\mathcal{L}_Xg^S
&=\mathcal{L}_X\pi^*g_{S_{a}}+\mathcal{L}_X\eta\otimes\eta+\eta\otimes\mathcal{L}_X\eta\\
&=\pi^*\bigl(\mathcal{L}_Vg_{S_{a}}\bigr)+\pi^*\beta\otimes\eta+\eta\otimes\pi^*\beta .
\end{aligned}
\]
\end{proof}

\subsection{The Sasaki metric is not Einstein}

\begin{theorem}\label{thm:einstein}
The Ricci curvature of $(T^1S_a,g^S)$ is
\[
\operatorname{Ric}^S
=-\tfrac32\,\pi^*g_{S_{a}}+\tfrac12\,\eta\otimes\eta
=-\tfrac32\,g^S+2\,\eta\otimes\eta .
\]
Equivalently, $\operatorname{Ric}^S$ is diagonal in the frame $(E_1,E_2,E_3)$, with eigenvalues
\[
\operatorname{Ric}^S(E_1,E_1)=\operatorname{Ric}^S(E_2,E_2)=-\tfrac32,
\qquad
\operatorname{Ric}^S(E_3,E_3)=\tfrac12 .
\]
The scalar curvature is $R^S=-\tfrac52$. In particular, $g^S$ is not an Einstein metric (and, since $\dim T^1S_a=3$, it does not have constant sectional curvature). These facts do not depend on $a$.
\end{theorem}

\begin{proof}
By \eqref{brackets}, the orthonormal frame $(E_1,E_2,E_3)$ satisfies
\[
[E_2,E_3]=\lambda_1E_1,\qquad
[E_3,E_1]=\lambda_2E_2,\qquad
[E_1,E_2]=\lambda_3E_3,
\qquad (\lambda_1,\lambda_2,\lambda_3)=(1,1,-1).
\]

\medskip
By Milnor's formula in \cite{Milnor} Lemma~4.1, the Ricci tensor is diagonal in this frame and
\[
\begin{aligned}
\operatorname{Ric}(E_1)&=2\mu_2\mu_3\,E_1,\\
\operatorname{Ric}(E_2)&=2\mu_1\mu_3\,E_2,\\
\operatorname{Ric}(E_3)&=2\mu_1\mu_2\,E_3,
\end{aligned}
\qquad
\mu_i=\tfrac12(\lambda_1+\lambda_2+\lambda_3)-\lambda_i .
\]
Here $\tfrac12(\lambda_1+\lambda_2+\lambda_3)=\tfrac12$, hence
\[
\mu_1=\mu_2=-\tfrac12,\qquad \mu_3=\tfrac32,
\]
and therefore
\[
\begin{aligned}
\operatorname{Ric}(E_1,E_1)&=\operatorname{Ric}(E_2,E_2)=2\cdot\bigl(-\tfrac12\bigr)\cdot\tfrac32=-\tfrac32,\\
\operatorname{Ric}(E_3,E_3)&=2\cdot\bigl(-\tfrac12\bigr)^2=\tfrac12 .
\end{aligned}
\]

\medskip
Since $f^1\otimes f^1+f^2\otimes f^2=\pi^*g_{S_{a}}$ and $f^3=\eta$, we obtain
\[
\operatorname{Ric}^S=-\tfrac32\,\pi^*g_{S_{a}}+\tfrac12\,\eta\otimes\eta ,
\]
and, taking the trace,
\[
R^S=-\tfrac32-\tfrac32+\tfrac12=-\tfrac52 .
\]

\medskip
If $g^S$ were Einstein, there would be a constant $\rho$ with $\operatorname{Ric}^S=\rho\,g^S$, hence
$\operatorname{Ric}(E_1,E_1)=\operatorname{Ric}(E_3,E_3)$. This is false, since $-\tfrac32\neq\tfrac12$. Finally, in dimension $3$ the Einstein condition is equivalent to constant sectional curvature, so $g^S$ does not have constant sectional curvature either.
\end{proof}
\subsection{Non-existence of Ricci solitons}

Recall that $(g^S,X,\lambda)$ is a Ricci soliton if $X$ is a smooth vector field and $\lambda\in\mathbb{R}$ satisfy
\begin{equation}\label{RS}
\operatorname{Ric}^S+\tfrac12\lambda\mathcal{L}_Xg^S=\lambda g^S .
\end{equation}
(The cases $\lambda<0$, $=0$, $>0$ are the shrinking, steady and expanding cases respectively, depending on the convention; any other normalization, for instance $\operatorname{Ric}+\lambda\mathcal{L}_Xg=\lambda g$, reduces to \eqref{RS} by replacing $X$ with $2X$.)

\begin{theorem}\label{thm:RS}
There is no Ricci soliton on $(T^1S_a,g^S)$: for every smooth vector field $X$ on $T^1S_a$ and every $\lambda\in\mathbb{R}$, equation \eqref{RS} has no solution. The same holds on $\pi^{-1}(U)$ for every open set $U\subset S_a$. In particular, there are no gradient, steady, shrinking or expanding Ricci solitons.
\end{theorem}

\begin{proof}
By Theorem~\ref{thm:einstein}, equation \eqref{RS} reads
\begin{equation}\label{RS2}
\mathcal{L}_Xg^S=h,
\end{equation}
where
\[
\begin{aligned}
h:=2\lambda\, g^S-2\operatorname{Ric}^S
&=(2\lambda+3)\,\pi^*g_{S_{a}}+(2\lambda-1)\,\eta\otimes\eta\\
&=\kappa\, g^S-4\,\eta\otimes\eta,
\qquad \kappa:=2\lambda+3 .
\end{aligned}
\]

\medskip
For $s\in\mathbb{R}/2\pi\mathbb{Z}$, let
\[
\psi_s(x,y,\theta)=(x,y,\theta+s).
\]
The coefficients of $g^S$ and $\eta$ do not depend on $\theta$, so $\psi_s^*g^S=g^S$ and $\psi_s^*\eta=\eta$; consequently $\psi_{s*}h=h$. For every vector field $X$,
\[
\mathcal{L}_{\psi_{s*}X}\,g^S=\psi_{s*}\bigl(\mathcal{L}_X(\psi_s^*g^S)\bigr)=\psi_{s*}\bigl(\mathcal{L}_Xg^S\bigr).
\]
Suppose $X$ satisfies \eqref{RS2} and put
\[
\bar X=\frac{1}{2\pi}\int_0^{2\pi}\psi_{s*}X\,ds .
\]
This field is smooth (the circle is compact) and $\theta$-invariant. Since $X\mapsto\mathcal{L}_Xg^S$ is linear,
\[
\mathcal{L}_{\bar X}g^S=\frac{1}{2\pi}\int_0^{2\pi}\psi_{s*}\bigl(\mathcal{L}_Xg^S\bigr)\,ds=\frac{1}{2\pi}\int_0^{2\pi}\psi_{s*}h\,ds=h .
\]
Hence we may assume that $X$ is $\theta$-invariant.

\medskip
With the notation of Lemma~\ref{lem:Lie}, equation \eqref{RS2} becomes
\[
\pi^*\bigl(\mathcal{L}_Vg_{S_{a}}-\kappa\, g_{S_{a}}\bigr)
+\eta\otimes\pi^*\beta+\pi^*\beta\otimes\eta
-(\kappa-4)\,\eta\otimes\eta=0 .
\]
Evaluate this on the orthonormal frame $(e_1,e_2,e_3)$ dual to $(e^1,e^2,e^3)$:
\[
\begin{aligned}
(a,b),\ a,b\in\{1,2\}:&\quad \mathcal{L}_Vg_{S_{a}}=\kappa\, g_{S_{a}},\\
(a,3),\ a\in\{1,2\}:&\quad \beta=0,\\
(3,3):&\quad \kappa=4 .
\end{aligned}
\]
Therefore
\[
\mathcal{L}_Vg_{S_{a}}=4\,g_{S_{a}}.
\]

\medskip
Let $\varphi_t$ be the local flow of $V$ and set $\Phi(t)=\varphi_t^*g_{S_{a}}$. Then
\[
\Phi'(t)=\varphi_t^*\bigl(\mathcal{L}_Vg_{S_{a}}\bigr)=4\,\Phi(t),
\]
so $\Phi(t)=e^{4t}g_{S_{a}}$, and the Gauss curvature of $\Phi(t)$ is $-e^{-4t}$. On the other hand, $\varphi_t$ is a local diffeomorphism, so the curvature of $\varphi_t^*g_{S_{a}}$ at a point $p$ equals the curvature of $g_{S_{a}}$ at $\varphi_t(p)$, namely $-1$. Thus $e^{-4t}=1$ for all small $t$, which is absurd.
\end{proof}

\subsection{Hyperbolic Ricci Solitons}
We use the definition of Fasihi-Ramandi--Azami--Choudhary \cite{FRAC}: a \emph{hyperbolic Ricci soliton} is a quadruple $(g,X,\lambda,\mu)$, where $X$ is a smooth vector field and $\lambda,\mu\in\mathbb{R}$, such that
\begin{equation}\label{HRS}
\operatorname{Ric}+\lambda\,\mathcal{L}_Xg+\tfrac12\,\mathcal{L}_X\bigl(\mathcal{L}_Xg\bigr)=\mu\, g .
\end{equation}

\begin{proposition}\label{prop:killing}
If $X$ is a Killing field, then \eqref{HRS} is equivalent to $\operatorname{Ric}=\mu g$. Hence no Killing field of $(T^1S_a,g^S)$, in particular no element of $\mathrm{span}\{K_0,K_1,K_2,K_3\}$, is the potential of a hyperbolic Ricci soliton.
\end{proposition}

\begin{proof}
If $\mathcal{L}_Xg=0$, then $\mathcal{L}_X(\mathcal{L}_Xg)=0$, and \eqref{HRS} reduces to $\operatorname{Ric}=\mu g$. This is impossible by Theorem~\ref{thm:einstein}.
\end{proof}

\begin{proposition}\label{prop:left}
There is no hyperbolic Ricci soliton whose potential is $X=x^1E_1+x^2E_2+x^3E_3$ with constants $x^i\in\mathbb{R}$.
\end{proposition}

\begin{proof}
Since the $x^i$ are constant, $\mathcal{L}_Xf^k=\iota_X\,df^k$. By \eqref{struct},
\[
\begin{aligned}
\mathcal{L}_Xf^1&=x^3f^2-x^2f^3,\\
\mathcal{L}_Xf^2&=x^1f^3-x^3f^1,\\
\mathcal{L}_Xf^3&=\gamma,
\end{aligned}
\qquad \gamma:=x^1f^2-x^2f^1 .
\]

\medskip
We have
\[
\mathcal{L}_Xg^S=\sum_{k=1}^3\bigl(\mathcal{L}_Xf^k\otimes f^k+f^k\otimes\mathcal{L}_Xf^k\bigr).
\]
The terms in $x^3$ cancel in pairs, and what remains is
\[
T:=\mathcal{L}_Xg^S=2\,\bigl(\gamma\otimes f^3+f^3\otimes\gamma\bigr).
\]
Next, $\mathcal{L}_X\gamma=x^1\mathcal{L}_Xf^2-x^2\mathcal{L}_Xf^1=\rho\, f^3-x^3\delta$, where
\[
\rho=(x^1)^2+(x^2)^2,\qquad \delta=x^1f^1+x^2f^2 .
\]
Using also $\mathcal{L}_Xf^3=\gamma$, we get
\[
\tfrac12\,\mathcal{L}_XT
=2\rho\,f^3\otimes f^3-x^3\bigl(\delta\otimes f^3+f^3\otimes\delta\bigr)+2\,\gamma\otimes\gamma .
\]

\medskip With
\[
\operatorname{Ric}^S=-\tfrac32\bigl(f^1\otimes f^1+f^2\otimes f^2\bigr)+\tfrac12\,f^3\otimes f^3,
\]
equation \eqref{HRS} gives, component by component,
\[
\begin{aligned}
(3,3):&\quad \tfrac12+2\rho=\mu,\\
(a,3),\ a=1,2:&\quad 2\lambda\gamma-x^3\delta=0,\\
(a,b),\ a,b\in\{1,2\}:&\quad 2\,\gamma\otimes\gamma=\bigl(\mu+\tfrac32\bigr)\,\mathrm{Id}.
\end{aligned}
\]

\medskip
In the last equation, the left-hand side has rank at most $1$, while the right-hand side has rank $2$ unless $\mu=-\tfrac32$. In that case $\gamma=0$, that is $x^1=x^2=0$ and $\rho=0$. But then the first equation gives $\mu=\tfrac12\neq-\tfrac32$, a contradiction.
\end{proof}

\begin{theorem}\label{thm:HRS}
Let $X$ be a $\theta$-invariant vector field, with $V,F,\beta$ as above, and put $S:=\mathcal{L}_Vg_{S_{a}}$. Then $(g^S,X,\lambda,\mu)$ is a hyperbolic Ricci soliton if and only if
\begin{align}
&\mu=\tfrac12, \label{E0}\\
&\lambda\,S+\tfrac12\,\mathcal{L}_VS+\beta\otimes\beta=2\,g_{S_{a}}, \label{E1}\\
&\mathcal{L}_V\beta=-2\lambda\,\beta . \label{E2}
\end{align}
Consequently, such a soliton is necessarily expanding with $\mu=\tfrac12$, and $V$ is neither zero, nor Killing, nor conformal for $g_{S_{a}}$.
\end{theorem}

\begin{proof}
By Lemma~\ref{lem:Lie},
\[
T:=\mathcal{L}_Xg^S=\pi^*S+\eta\otimes\beta+\beta\otimes\eta
\]
(we omit $\pi^*$ in front of $\beta$). Since $\mathcal{L}_X\eta=\beta$, $\mathcal{L}_X\beta=\mathcal{L}_V\beta$ and $\mathcal{L}_X\pi^*S=\pi^*\mathcal{L}_VS$, we get
\[
\mathcal{L}_XT=\pi^*\mathcal{L}_VS+2\,\beta\otimes\beta+\eta\otimes\mathcal{L}_V\beta+\mathcal{L}_V\beta\otimes\eta .
\]

\medskip
With $\operatorname{Ric}^S=-\tfrac32\,\pi^*g_{S_{a}}+\tfrac12\,\eta\otimes\eta$, equation \eqref{HRS} reads
\[
\begin{aligned}
&\pi^*\Bigl(-\tfrac32\,g_{S_{a}}+\lambda S+\tfrac12\,\mathcal{L}_VS+\beta\otimes\beta-\mu\, g_{S_{a}}\Bigr)\\
&\quad+\eta\otimes\Bigl(\lambda\beta+\tfrac12\,\mathcal{L}_V\beta\Bigr)+\Bigl(\lambda\beta+\tfrac12\,\mathcal{L}_V\beta\Bigr)\otimes\eta\\
&\quad+\Bigl(\tfrac12-\mu\Bigr)\,\eta\otimes\eta=0 .
\end{aligned}
\]
Setting the $(a,b)$, $(a,3)$ and $(3,3)$ components in the coframe $(e^i)$ to zero separately, we obtain respectively \eqref{E1} (using $\mu=\tfrac12$, so that $-\tfrac32-\mu=-2$), \eqref{E2} and \eqref{E0}.

\medskip
If $V=0$ or if $V$ is Killing, then $S=0$ and $\mathcal{L}_VS=0$, so \eqref{E1} gives $\beta\otimes\beta=2\,g_{S_{a}}$. This is impossible, since the left-hand side has rank at most $1$ and the right-hand side has rank $2$.

If $V$ is conformal, then $S=\varphi\, g_{S_{a}}$ with $\varphi=\operatorname{div}V$, and $\mathcal{L}_VS=(V\varphi+\varphi^2)\,g_{S_{a}}$. Equation \eqref{E1} gives
\[
\beta\otimes\beta=\Bigl(2-\lambda\varphi-\tfrac12\bigl(V\varphi+\varphi^2\bigr)\Bigr)\,g_{S_{a}},
\]
which forces $\beta=0$ for the same rank reason. Then $0=d\beta=\varphi\,\mathrm{vol}_B$, so $\varphi=0$, i.e.\ $V$ is Killing, and \eqref{E1} reduces to $0=2\,g_{S_{a}}$: absurd.
\end{proof}

\begin{remark}
The results above show that a hyperbolic Ricci soliton on $(T^1S_a,g^S)$ cannot have as potential a Killing field, a left-invariant field, or a $\theta$-invariant field whose projection is conformal. For $\theta$-invariant potentials, the problem is reduced to the system \eqref{E1}--\eqref{E2} on the hyperbolic surface $S_a$. Condition \eqref{E1} says that the tensor
\[
2\,g_{S_{a}}-\lambda S-\tfrac12\,\mathcal{L}_VS
\]
has rank at most $1$ and equals $\beta\otimes\beta$, where $d\beta=\operatorname{div}(V)\,\mathrm{vol}_B$. The complete classification for an arbitrary potential $X$ (not $\theta$-invariant) is not treated here: since equation \eqref{HRS} is nonlinear in $X$, the averaging argument from the proof of Theorem~\ref{thm:RS} does not apply.
\end{remark}
\section{$\eta-$Ricci Solitons}

\begin{theorem}\label{thm:etaRS}
Let $\eta=d\theta+c\,dx$ be the contact form of $(T^1S_a,g^S)$. Then $(g^S,X,\lambda,\mu)$ is an $\eta$-Ricci soliton if and only if
\[
\lambda=\tfrac32,\qquad \mu=-2,\qquad X\in\mathfrak{Kill}(T^1S_a,g^S)=\operatorname{span}_{\mathbb{R}}\{K_0,K_1,K_2,K_3\}.
\]
In particular, the constants $(\lambda,\mu)$ are uniquely determined, the soliton is expanding, and every potential field is a Killing field. There is no $\eta$-Ricci soliton with a non-Killing potential.
\end{theorem}

\begin{proof}
By Theorem~\ref{thm:einstein}, $2\operatorname{Ric}^S=-3\,g^S+4\,\eta\otimes\eta$. Hence \eqref{etaRS} is equivalent to
\begin{equation}\label{etaRS2}
\mathcal{L}_Xg^S=h,\qquad h:=p\,g^S+q\,\eta\otimes\eta,\qquad p:=3-2\lambda,\quad q:=-(4+2\mu).
\end{equation}

\medskip
Let $\psi_s(x,y,\theta)=(x,y,\theta+s)$. As in the proof of Theorem~\ref{thm:RS}, $\psi_s^*g^S=g^S$ and $\psi_s^*\eta=\eta$, so that $\psi_{s*}h=h$ and
\[
\mathcal{L}_{\psi_{s*}X}\,g^S=\psi_{s*}\bigl(\mathcal{L}_Xg^S\bigr).
\]
If $X$ satisfies \eqref{etaRS2}, the field
\[
\bar X=\frac{1}{2\pi}\int_0^{2\pi}\psi_{s*}X\,ds
\]
is smooth and $\theta$-invariant, and by linearity of $X\mapsto\mathcal{L}_Xg^S$ it satisfies $\mathcal{L}_{\bar X}g^S=h$.

\medskip
Apply Lemma~\ref{lem:Lie} to $\bar X$, with $V=\pi_*\bar X$ and $\beta$ as there. Since $g^S=\pi^*g_{S_{a}}+\eta\otimes\eta$, equation \eqref{etaRS2} reads
\[
\pi^*\bigl(\mathcal{L}_Vg_{S_{a}}\bigr)+\eta\otimes\pi^*\beta+\pi^*\beta\otimes\eta
=p\,\pi^*g_{S_{a}}+(p+q)\,\eta\otimes\eta .
\]
Evaluating on the orthonormal frame $(e_1,e_2,e_3)$ dual to $(e^1,e^2,e^3)$:
\[
\begin{aligned}
(a,b),\ a,b\in\{1,2\}:&\quad \mathcal{L}_Vg_{S_{a}}=p\,g_{S_{a}},\\
(a,3),\ a\in\{1,2\}:&\quad \beta=0,\\
(3,3):&\quad p+q=0 .
\end{aligned}
\]

\medskip
Let $\varphi_t$ be the local flow of $V$ and set $\Phi(t)=\varphi_t^*g_{S_{a}}$. Then
\[
\Phi'(t)=\varphi_t^*\bigl(\mathcal{L}_Vg_{S_{a}}\bigr)=p\,\Phi(t),
\]
so $\Phi(t)=e^{pt}g_{S_{a}}$, whose Gauss curvature is $-e^{-pt}$. On the other hand, $\varphi_t$ is a local diffeomorphism and $g_{S_{a}}$ has Gauss curvature $-1$, so the curvature of $\Phi(t)$ equals $-1$. Thus $e^{-pt}=1$ for all small $t$, hence $p=0$, and then $q=-p=0$.

\medskip
From $p=q=0$ we get $3-2\lambda=0$ and $4+2\mu=0$, that is,
\[
\lambda=\tfrac32,\qquad \mu=-2 .
\]
Moreover $h=0$, so the original field $X$ satisfies $\mathcal{L}_Xg^S=0$: it is a Killing field, i.e.\ $X\in\operatorname{span}_{\mathbb{R}}\{K_0,K_1,K_2,K_3\}$ by the structure theorem for $\mathfrak{Kill}(T^1S_a,g^S)$.

\medskip
Let $X$ be a Killing field, $\lambda=\tfrac32$ and $\mu=-2$. Then
\[
\begin{aligned}
\mathcal{L}_Xg^S+2\operatorname{Ric}^S+2\lambda\, g^S+2\mu\,\eta\otimes\eta
&=0+\bigl(-3\,g^S+4\,\eta\otimes\eta\bigr)+3\,g^S-4\,\eta\otimes\eta\\
&=0 .
\end{aligned}
\]
\end{proof}
\section{$\eta-$hyperbolic Ricci Solitons}
In this section we will define the notion of $\eta-$ Hyperbolic Ricci Soliton for the manifold of dimension $2n+1.$ And we will describe the $\eta-$hyperbolic Ricci Soliton for the Sasaki metric.

\begin{definition}[$\eta$-hyperbolic Ricci soliton]\label{def:etaHRS}
Let $(M,g)$ be a Riemannian manifold and $\eta$ a $1$-form on $M$. A quintuple $(g,X,\lambda,\mu,\nu)$, where $X$ is a smooth vector field and $\lambda,\mu,\nu\in\mathbb{R}$, is called an \emph{$\eta$-hyperbolic Ricci soliton} if
\begin{equation}\label{etaHRS}
\operatorname{Ric}+\lambda\,\mathcal{L}_Xg+\tfrac12\,\mathcal{L}_X\bigl(\mathcal{L}_Xg\bigr)=\mu\,g+\nu\,\eta\otimes\eta .
\end{equation}
For $\nu=0$ one recovers the hyperbolic Ricci solitons of \cite{FRAC}. Formally, dropping the second-order term and taking $\lambda=\tfrac12$ gives the $\eta$-Ricci soliton equation of Cho--Kimura \cite{ChoKimura}, written as $\operatorname{Ric}+\tfrac12\mathcal{L}_Xg=\mu g+\nu\,\eta\otimes\eta$. If $X$ is a Killing field, \eqref{etaHRS} reduces to $\operatorname{Ric}=\mu g+\nu\,\eta\otimes\eta$, i.e.\ to the $\eta$-Einstein condition; such solitons are called \emph{trivial}.
\end{definition}
In what follows, $\eta=d\theta+c\,dx$ is the contact form of $(T^1S_a,g^S)$, and for a $\theta$-invariant field $X$ we use $V,F,\beta$ and $S:=\mathcal{L}_Vg_{S_{a}}$ as in Lemma~\ref{lem:Lie}.

\begin{theorem}[$\eta$-hyperbolic Ricci solitons of the Sasaki metric]\label{thm:etaHRS}
Let $(g^S,X,\lambda,\mu,\nu)$ be given.
\begin{enumerate}
\item[\textup{(a)}] If $X$ is a Killing field, then $(g^S,X,\lambda,\mu,\nu)$ is an $\eta$-hyperbolic Ricci soliton if and only if
\[
\mu=-\tfrac32,\qquad \nu=2,\qquad \lambda\in\mathbb{R}\ \text{arbitrary}.
\]
\item[\textup{(b)}] Let $X=x^1E_1+x^2E_2+x^3E_3$ with constants $x^i\in\mathbb{R}$. Then $(g^S,X,\lambda,\mu,\nu)$ is an $\eta$-hyperbolic Ricci soliton if and only if
\[
x^1=x^2=0\ (\text{so } X=x^3K_0),\qquad \mu=-\tfrac32,\qquad \nu=2 .
\]
\item[\textup{(c)}] Let $X$ be $\theta$-invariant. Then $(g^S,X,\lambda,\mu,\nu)$ is an $\eta$-hyperbolic Ricci soliton if and only if
\begin{align}
&\mu+\nu=\tfrac12, \label{R0}\\
&\lambda\,S+\tfrac12\,\mathcal{L}_VS+\beta\otimes\beta=\bigl(\mu+\tfrac32\bigr)\,g_{S_{a}}, \label{R1}\\
&\mathcal{L}_V\beta=-2\lambda\,\beta . \label{R2}
\end{align}
\item[\textup{(d)}] Let $X$ be $\theta$-invariant and assume that $V$ is zero, Killing, or conformal for $g_{S_{a}}$. If $(g^S,X,\lambda,\mu,\nu)$ is an $\eta$-hyperbolic Ricci soliton, then $X$ is a Killing field, $\mu=-\tfrac32$ and $\nu=2$.
\end{enumerate}
\end{theorem}

\begin{proof}
\textbf{(a)} If $\mathcal{L}_Xg^S=0$, then $\mathcal{L}_X(\mathcal{L}_Xg^S)=0$, and \eqref{etaHRS} reads $\operatorname{Ric}^S=\mu g^S+\nu\,\eta\otimes\eta$. By Theorem~\ref{thm:einstein},
\[
\bigl(\mu+\tfrac32\bigr)\,g^S+(\nu-2)\,\eta\otimes\eta=0 .
\]
Evaluating on $(E_1,E_1)$ gives $\mu+\tfrac32=0$; evaluating on $(E_3,E_3)$ gives $\mu+\tfrac32+\nu-2=0$, hence $\nu=2$. Conversely, for these values the equation holds for every $\lambda$ and every Killing field $X$.

\medskip
\textbf{(b)} Since the $x^i$ are constant, $\mathcal{L}_Xf^k=\iota_Xdf^k$, and by \eqref{struct}
\[
\mathcal{L}_Xf^1=x^3f^2-x^2f^3,\quad \mathcal{L}_Xf^2=x^1f^3-x^3f^1,\quad \mathcal{L}_Xf^3=\gamma,\qquad \gamma:=x^1f^2-x^2f^1 .
\]
Hence $T:=\mathcal{L}_Xg^S=\sum_k\bigl(\mathcal{L}_Xf^k\otimes f^k+f^k\otimes\mathcal{L}_Xf^k\bigr)=2(\gamma\otimes f^3+f^3\otimes\gamma)$, the terms in $x^3$ cancelling in pairs. Moreover $\mathcal{L}_X\gamma=\rho f^3-x^3\delta$ with $\rho=(x^1)^2+(x^2)^2$ and $\delta=x^1f^1+x^2f^2$, so
\[
\tfrac12\mathcal{L}_XT=2\rho\,f^3\otimes f^3-x^3(\delta\otimes f^3+f^3\otimes\delta)+2\,\gamma\otimes\gamma .
\]
With $\operatorname{Ric}^S=-\tfrac32(f^1\otimes f^1+f^2\otimes f^2)+\tfrac12f^3\otimes f^3$, equation \eqref{etaHRS} gives
\[
\begin{aligned}
(3,3):&\quad \tfrac12+2\rho=\mu+\nu,\\
(a,3),\ a=1,2:&\quad 2\lambda\gamma-x^3\delta=0,\\
(a,b),\ a,b\in\{1,2\}:&\quad 2\,\gamma\otimes\gamma=\bigl(\mu+\tfrac32\bigr)\mathrm{Id}.
\end{aligned}
\]
In the last equation the left-hand side has rank at most $1$, while the right-hand side has rank $2$ unless $\mu=-\tfrac32$; then $\gamma=0$, i.e.\ $x^1=x^2=0$ and $\rho=0$. The first equation gives $\tfrac12=-\tfrac32+\nu$, so $\nu=2$, and the second equation holds trivially. Conversely, $X=x^3E_3=x^3K_0$ is Killing and (a) applies.

\medskip
\textbf{(c)} By Lemma~\ref{lem:Lie}, $T:=\mathcal{L}_Xg^S=\pi^*S+\eta\otimes\beta+\beta\otimes\eta$ (we omit $\pi^*$ before $\beta$). Since $\mathcal{L}_X\eta=\beta$, $\mathcal{L}_X\beta=\mathcal{L}_V\beta$ and $\mathcal{L}_X\pi^*S=\pi^*\mathcal{L}_VS$,
\[
\mathcal{L}_XT=\pi^*\mathcal{L}_VS+2\,\beta\otimes\beta+\eta\otimes\mathcal{L}_V\beta+\mathcal{L}_V\beta\otimes\eta .
\]
Using $\operatorname{Ric}^S=-\tfrac32\pi^*g_{S_{a}}+\tfrac12\eta\otimes\eta$ and $g^S=\pi^*g_{S_{a}}+\eta\otimes\eta$, equation \eqref{etaHRS} becomes
\[
\begin{aligned}
&\pi^*\Bigl(-\tfrac32g_{S_{a}}+\lambda S+\tfrac12\mathcal{L}_VS+\beta\otimes\beta-\mu\,g_{S_{a}}\Bigr)\\
&\quad+\eta\otimes\Bigl(\lambda\beta+\tfrac12\mathcal{L}_V\beta\Bigr)+\Bigl(\lambda\beta+\tfrac12\mathcal{L}_V\beta\Bigr)\otimes\eta\\
&\quad+\Bigl(\tfrac12-\mu-\nu\Bigr)\eta\otimes\eta=0 .
\end{aligned}
\]
Setting the $(a,b)$, $(a,3)$ and $(3,3)$ components in the coframe $(e^i)$ to zero gives \eqref{R1}, \eqref{R2} and \eqref{R0}.

\medskip
\textbf{(d)} Put $a:=\mu+\tfrac32$. If $V=0$ or $V$ is Killing, then $S=0$ and \eqref{R1} reads $\beta\otimes\beta=a\,g_{S_{a}}$. The left side has rank at most $1$ and the right side has rank $2$ unless $a=0$; hence $a=0$ and $\beta=0$. If $V$ is conformal, $S=\varphi\,g_{S_{a}}$ with $\varphi=\operatorname{div}V$, and \eqref{R1} gives
\[
\beta\otimes\beta=\Bigl(a-\lambda\varphi-\tfrac12\bigl(V\varphi+\varphi^2\bigr)\Bigr)g_{S_{a}},
\]
which forces $\beta=0$ for the same rank reason. Then $0=d\beta=\varphi\,\mathrm{vol}_{S_a}$, so $\varphi=0$, $V$ is Killing, and we are back in the previous case with $a=0$. In all cases $S=0$ and $\beta=0$, so $\mathcal{L}_Xg^S=\pi^*S+\eta\otimes\beta+\beta\otimes\eta=0$: $X$ is Killing. Then (a) gives $\mu=-\tfrac32$ and $\nu=2$.
\end{proof}

\begin{proposition}\label{prop:shear}
Let $X$ be $\theta$-invariant with $V=\rho(y)\,\partial_x$ on an open connected set. If $(g^S,X,\lambda,\mu,\nu)$ is an $\eta$-hyperbolic Ricci soliton, then $\rho'=0$ (so $V=\rho_0\partial_x$ is Killing), and therefore $X$ is a Killing field with $\mu=-\tfrac32$, $\nu=2$.
\end{proposition}

\begin{proof}
Since $\partial_x$ is Killing and $w$ depends only on $y$, we have $\mathcal{L}_Vdx=\rho'\,dy$, $\mathcal{L}_Vdy=0$ and $V(w\rho')=0$. Hence
\[
S=w\rho'\,(dx\otimes dy+dy\otimes dx)=\rho'\,(e^1\otimes e^2+e^2\otimes e^1),\qquad
\mathcal{L}_VS=2\rho'^2\,e^2\otimes e^2 .
\]
Put $a=\mu+\tfrac32$. By \eqref{R1}, $\beta\otimes\beta=a\,g_{S_{a}}-\lambda S-\tfrac12\mathcal{L}_VS$, whose matrix in the frame $(e_1,e_2)$ is
\[
\begin{pmatrix} a & -\lambda\rho'\\ -\lambda\rho' & a-\rho'^2\end{pmatrix}.
\]
This matrix is positive semidefinite of rank at most $1$, so $a\ge0$, $a-\rho'^2\ge0$ and $a(a-\rho'^2)=\lambda^2\rho'^2$.

\emph{Case $a=0$.} Then $-\rho'^2\ge0$, so $\rho'=0$.

\emph{Case $a>0$.} Then $a^2=(a+\lambda^2)\rho'^2$, so $\rho'^2=a^2/(a+\lambda^2)$ is constant, and $\beta=\varepsilon\bigl(\sqrt a\,e^1-\tfrac{\lambda\rho'}{\sqrt a}\,e^2\bigr)$ with a locally constant sign $\varepsilon=\pm1$. Since $e^1=\sqrt w\,dx$, $e^2=\sqrt w\,dy$ and $(\sqrt w)'=-c\sqrt w$,
\[
d\beta=\varepsilon\sqrt a\,d(\sqrt w)\wedge dx-\varepsilon\tfrac{\lambda\rho'}{\sqrt a}\,d(\sqrt w)\wedge dy
=\varepsilon\sqrt a\,c\sqrt w\,dx\wedge dy,
\]
which is nonzero for $y\neq a/2$. But $d\beta=\operatorname{div}(V)\,\mathrm{vol}_{S_a}=0$ because $\operatorname{div}(\rho(y)\partial_x)=\tfrac1w\partial_x(w\rho)=0$. This is a contradiction.

Hence $\rho'=0$, $V=\rho_0\partial_x$ is Killing, and Theorem~\ref{thm:etaHRS}\,(d) concludes.
\end{proof}



\begin{thebibliography}{99}

\bibitem{Ahlfors}
L.~Ahlfors, \emph{Complex Analysis}, 3rd ed., McGraw-Hill, 1979.

\bibitem{Bonola}
R.~Bonola, \emph{Non-Euclidean Geometry: A Critical and Historical Study of its Development}, Dover, 1955.

\bibitem{doCarmo}
M.~do Carmo, \emph{Differential Geometry of Curves and Surfaces}, Prentice-Hall, 1976.

\bibitem{FRAC}
G.~Fasihi-Ramandi, S.~Azami and M.~A.~Choudhary, \emph{On hyperbolic Ricci solitons}, Chaos Solitons Fractals \textbf{193} (2025), 116095. \url{https://doi.org/10.1016/j.chaos.2025.116095}

\bibitem{Hamilton}
R.~S.~Hamilton, \emph{Three-manifolds with positive Ricci curvature}, J.\ Differential Geom.\ \textbf{17} (1982), 255--306.

\bibitem{KongLiu}
D.-X.~Kong and K.~Liu, \emph{Wave character of metrics and hyperbolic geometric flow}, J.\ Math.\ Phys.\ \textbf{48} (2007), 103508.

\bibitem{Milnor}
J.~Milnor, \emph{Curvatures of left invariant metrics on Lie groups}, Adv.\ Math.\ \textbf{21} (1976), 293--329.

\bibitem{Nemeth}
S.~Z.~N\'emeth, \emph{Five kinds of monotone vector fields}, Pure Math.\ Appl.\ \textbf{9} (1998), 417--428.

\bibitem{Perelman}
G.~Perelman, \emph{The entropy formula for the Ricci flow and its geometric applications}, arXiv:math/0211159 (2002).

\bibitem{Petersen}
P.~Petersen, \emph{Riemannian Geometry}, 3rd ed., Graduate Texts in Mathematics, vol.~171, Springer, 2016.
\bibitem{ChoKimura}
J.~T.~Cho and M.~Kimura, \emph{Ricci solitons and real hypersurfaces in a complex space form}, Tohoku Math.\ J.\ \textbf{61} (2009), 205--212.
\bibitem{Sasaki}
S.~Sasaki, \emph{On the differential geometry of tangent bundles of Riemannian manifolds}, Tohoku Math.\ J.\ \textbf{10} (1958), 338--354.

\bibitem{Tashiro}
Y.~Tashiro, \emph{On contact structures of tangent sphere bundles}, Tohoku Math.\ J.\ \textbf{21} (1969), 117--143.

\bibitem{Kowalski}
O.~Kowalski, \emph{Curvature of the induced Riemannian metric on the tangent bundle of a Riemannian manifold}, J.\ Reine Angew.\ Math.\ \textbf{250} (1971), 124--129.

\bibitem{Blaga}
A.~M.~Blaga, \emph{$\eta$-Ricci solitons on para-Kenmotsu manifolds}, Balkan J.\ Geom.\ Appl.\ \textbf{20} (2015), 1--13.

\bibitem{BoyerGalicki}
C.~P.~Boyer and K.~Galicki, \emph{Sasakian Geometry}, Oxford University Press, 2008.

\bibitem{Besse}
A.~L.~Besse, \emph{Einstein Manifolds}, Springer, 1987.
\bibitem {am}
 M. Lyubich, notes/livre sur les modèles du plan hyperbolique
\end{thebibliography}
\end{document}